\documentclass[11pt]{amsart}
\usepackage{amsfonts,amssymb,epsfig}
\usepackage[latin1]{inputenc}
\usepackage[all]{xy}
\usepackage{amsmath}
\usepackage{bm}
\usepackage{delarray}
\usepackage{amssymb}

\newtheorem{thm}{Theorem}[section]

\newtheorem{prop}[thm]{Proposition}
\theoremstyle{definition}
\newtheorem{defn}[thm]{Definition}
\theoremstyle{remark}
\newtheorem{rem}[thm]{Remark}
\numberwithin{equation}{section}

\renewcommand{\epsilon}{\varepsilon}

\begin{document}

\title[Symplectic linearization of semisimple Lie algebra actions]{ A note on symplectic and Poisson linearization of semisimple Lie algebra actions}

\author{Eva Miranda}
\address{ Eva Miranda,
Departament de Matem\`{a}tica Aplicada I, Universitat Polit\`{e}cnica de Catalunya, Barcelona, Spain, \it{e-mail: eva.miranda@upc.edu}}
\thanks{Eva Miranda is partially supported by Ministerio de Econom\'{i}a y Competitividad project ``Geometr\'{\i}a Algebraica, Simpl\'{e}ctica, Aritm\'{e}tica y Aplicaciones'' with reference MTM2012-38122-C03-01/FEDER and by the European Science Foundation network CAST}

\date{\today}%

\begin{abstract}
In this note we prove that an analytic symplectic action of a semisimple Lie algebra can be locally linearized in Darboux coordinates. This result yields simultaneous analytic linearization for Hamiltonian vector fields in a neighbourhood of a common zero.
We also provide an example of smooth non-linearizable Hamiltonian action
with semisimple linear part.  The smooth analogue only holds if the semisimple Lie algebra is of compact type. An analytic equivariant $b$-Darboux theorem for $b$-Poisson manifolds and  an analytic equivariant Weinstein splitting theorem  for general Poisson manifolds are also obtained in the Poisson setting.

\end{abstract}
\maketitle

\section{Introduction}

A classical result due to Bochner \cite{bo} establishes that a compact Lie group action on a smooth manifold is locally
equivalent, in the neighbourhood of a fixed point, to the linear
action. This result holds in the $\mathcal{C}^k$ category.  It is interesting to comprehend to which extent  we can expect a
similar behaviour for non-compact groups.

 As observed in \cite{guilleminsternberg} if the Lie group is connected  the linearization problem can be formulated in the following terms:  find a linear system of coordinates for the vector fields corresponding to the one parameter subgroups of $G$ or, more generally, consider the representation of a Lie algebra and find coordinates on the manifold that simultaneously linearize the vector fields in the image of the representation that vanish at a point. This is the point of view that we adopt in this note when we refer to \emph{linearization}.

 In the formal and analytic
case, the existence of coordinates that linearize the action is
related to a cohomological equation that can be solved when
the Lie group under consideration is semisimple \cite{kushnirenko},
\cite{guilleminsternberg}. Guillemin and Sternberg considered the
problem also in the $\mathcal C^{\infty}$ setting. At the end of  \cite{guilleminsternberg} the celebrated example of
non-linearizable action of $\mathfrak{sl}(2,\mathbb R)$ on $\mathbb R^3$ is constructed.
The example is based on a perturbation using the radial vector field
with flat coefficients. This example has been an outstanding one in the literature
since it yields the construction of many other examples with deep
geometrical meaning (for instance an example of non-stable Poisson
structure proposed by Weinstein in \cite{weinstein}). When the
semisimple Lie algebras are of compact type,
linearization of the action can be obtained simply by  combining local integration to the
action of a compact Lie group $G$ and Bochner's theorem
to linearize the action of the group. This, in its turn, leads to the
linearization of the Lie algebra action \cite{ruiphilippe}.

Another area in which linearization becomes a useful technique is
the one of Hamiltonian systems. When the Hamiltonian system is given by a symplectic action of a
compact Lie group fixing a point, it is known thanks to the
equivariant version of Darboux theorem (\cite{wein1},
\cite{chaperon}) that the action of the group can be linearized in a
neighbourhood of the fixed point in Darboux coordinates. It is noteworthy to understand to which extent this
situation can be achieved in more general contexts than the
compact one. If the Hamiltonian system is given by a complete integrable system
then there is an abelian symplectic action associated to this
integrable system. When  the integrable systems in local coordinates has  an associated \lq\lq linear part" given by a Cartan subalgebra, we
are led to non-degenerate singularities \cite{eliassonthesis}.

Complete integrable systems in a neighbourhood of a non-degenerate
singular point are equivalent to their linear models, as proved in \cite{eliassonelliptic}
,\cite{eliassonthesis},\cite{evathesis} and \cite{mirandazungequiv}, and thus  the Hamiltonian system is equivalent
to the linear one. This result establishes normal forms for
integrable systems in the neighbourhood of a singular non-degenerate
point and, in particular, provides a result of simultaneous linearization of Hamiltonian vector fields in a neighbourhood of a common zero. The next problem to be considered in this list is the case of
Hamiltonian systems which have linear part of semisimple type as proposed by Eliasson in \cite{eliassonelliptic}. In
the formal/analytic setting, according to the result of Guillemin
and Sternberg \cite{guilleminsternberg} and Kushnirenko
\cite{kushnirenko}, those systems are equivalent to the linear one
if the symplectic form is not taken into account. In this  note we prove that
not only the Hamiltonian vector fields can be linearized but also
that they can be linearized in Darboux coordinates.

Following the spirit of Guillemin and Sternberg, we prove that if a
symplectic Lie algebra action of semisimple type fixes a point there
exist  analytic Darboux coordinates in which the analytic vector
fields generating the Lie algebra action are linear. The result also
holds for complex analytic Lie algebra actions on complex analytic
manifolds.  The proof relies on a combination of Moser's path method and the complexification together with Weyl unitarian trick. We also construct an example
of Hamiltonian system with linear part of semisimple type and which
is not $\mathcal{C}^{\infty}$-linearizable. These results  are exported to the context of Poisson manifolds, starting with the case of $b$-Poisson manifolds for which Moser's path method works well. We end up this note proving a linearization theorem for the Lie algebra action in Weinstein's splitting coordinates for a Poisson structure.

In this note we study local rigidity problems for semisimple Lie
algebra actions.
 In \cite{mirandatenerife} we studied the global rigidity problem for compact group actions
 on compact symplectic manifolds proving that  two close smooth actions on a compact symplectic manifold $(M,\omega)$ are conjugate by a diffeomorphism preserving the
 symplectic form. In \cite{mirandamonnierzung} the contact and Poisson case were considered.    The analogous question in the
  analytic context for semisimple Lie groups which are not of
  compact type does not apply in the symplectic context since in view of \cite{delzant}\footnote{ Thanks to Ghani Zeghib for pointing this out.}  there cannot exist actions
  of semisimple Lie groups of non-compact type on a compact
  symplectic manifold.

{\bf{Organization of this note:}} In Section 2 we prove that
given a  real analytic symplectic action of a semisimple Lie
algebra, this action can be linearized in real  analytic Darboux
coordinates in a neighbourhood of a fixed point. This result also
holds for analytic complex manifolds and complex analytic actions of
semisimple Lie algebras. In section 3 we  provide a
counterexample that shows that a linearization result does not hold
in general for any smooth semisimple Lie algebra action. In section 4 we consider the Poisson analogues of Section 2 with a special focus on  the case of $b$-Poisson manifolds.

 {\bf{Acknowledgements}}: I am deeply thankful to Hakan Eliasson
for proposing me this problem and encouraging me to write down this note.
Many thanks also to Ghani Zheghib for many useful comments and
for drawing my attention to  \cite{delzant}. Thanks to David Martínez-Torres for his comments and corrections on a first version of this note. Many thanks to Marco Castrillón for very useful discussions concerning Section 3.

\section{Symplectic Linearization for analytic semisimple Lie algebra actions}

Let $\mathfrak g$ be a {\emph semisimple Lie algebra} and let $\rho:
\mathfrak g\longrightarrow L_{analytic}$ stand for a representation of
$\mathfrak g$ in the algebra of real (or complex) analytic vector fields
on a real (or complex) analytic manifold $M$.

We say that  $p\in M$  is a \textbf{fixed point} for $\rho$ if  the vector fields in $\rho(\mathfrak{g})$ vanish at $p$. We say that $\rho$ can be \textbf{linearized} in a neighborhood of a fixed point if there exist local coordinates in a neighbourhood of $p$ such that the vector fields in the image of $\rho$ can be simultaneously linearized (i.e, $\rho$ is \textbf{equivalent to a linear representation}).

 Guillemin and Sternberg \cite{guilleminsternberg} and Kushnirenko
\cite{kushnirenko} proved the following.

\begin{thm}[Guillemin-Sternberg, Kushnirenko]\label{thm:gs} The representation  $\rho: \mathfrak
g\longrightarrow L_{analytic}$ with $\mathfrak g$  semisimple  is locally equivalent, via an
analytic diffeomorphism, to a linear representation of $\mathfrak g$ in
a neighbourhood of a fixed point for $\rho$.
\end{thm}

When the representation is done by Hamiltonian vector fields (locally symplectic), the analytic diffeomorphism that gives the
equivalence of the initial representation to the linear
representation can be chosen to take the initial symplectic form to
the Darboux one. Namely,

\begin{thm}\label{thm:main}  Let $\mathfrak g$ be a semisimple Lie algebra and let $(M,\omega)$ be a (real or complex) analytic symplectic
manifold. Let   $\rho: \mathfrak g\longrightarrow L_{analytic}$  be a
representation by  analytic symplectic vector fields. Then there
exist local analytic coordinates $(x_1,y_1,\dots,x_n,y_n)$ in a
neighbourhood of a fixed point $p$ for $\rho$ such that $\rho$
is a linear representation and $\omega$ can be written as, $$\omega=\sum_{i=1}^n dx_i\wedge dy_i.$$
\end{thm}

\begin{proof}

Let $p\in M$ be a fixed point for $\rho$. By virtue of Darboux
theorem we may assume that $\omega$ is  $\omega_0=\sum_{i=1}^n dx_i\wedge dy_i$ for analytic coordinates
$(x_1,y_1,\dots,x_n,y_n)$ centered at $p$. Consider the linearized action $\rho^{(1)}$. Observe that if $\rho$ preserves
$\omega_0$ so does $\rho^{(1)}$.
By virtue of theorem
\ref{thm:gs} in the new local coordinates, we may assume that the action $\rho$ is linear but, in principle, the symplectic form in the new coordinates is not of Darboux type.
The new symplectic form is $\omega_1=\phi^*(\omega_0)$ (with $\phi$ the diffeomorphism provided by Theorem \ref{thm:gs}).

 It remains to show that we can take the analytic symplectic
form $\omega_1$ to  $\omega_0$ in such a way the representation
$\rho$ is preserved. For that purpose we apply the path method \cite{moser} for analytic symplectic
structures.

 Since $\omega_1$ belongs to the same
 cohomology class as $\omega_0$ we can write  $\omega_1=\omega_0+d\alpha$
 for an analytic $1$-form $\alpha$. Let us consider the linear path of analytic 2-forms
$\omega_t=t\omega_1+(1-t) \omega_0, \quad  t\in [0,1]$ which is symplectic in a neighbourhood of $p$. Notice that this path is invariant under the action $\rho$ since
both $\omega_0$ and $\omega_1$ are invariant by this action.
 Let $X_t$ be the analytic vector field defined by the Moser's equation,
\begin{equation}
\label{eqn:cohom} i_{X_t}\omega_t=-\alpha
\end{equation}
 Let $X_t^c$ be the complexified holomorphic vector field on
 $\mathbb C^{2n}$
associated to the real analytic vector field $X_t$. We use this vector field to find an analytic
diffeomorphism taking $\omega_0$ to $\omega_1$ and preserving the
linear character of  $\rho$. Let ${\rho}^c$ be the extended action of the complex semisimple Lie
algebra $\mathfrak g\otimes \mathbb C$ associated to $\rho$ defined as
follows. By choosing a sufficiently small disk centered at $p$ the
linear analytic vector fields in the image of $\rho^{(1)}$ can be
analytically extended to a small disk in $\mathbb C^{2n}$ to get a
representation by holomorphic vector fields. We then define the
linear action of the complex semisimple Lie algebra $\mathfrak g\otimes
\mathbb C$  via the formula
$$\rho^c(\sqrt{-1}x):=\sqrt{-1}\rho(x).$$

Let $\mathfrak h$ be a compact real form for the complex Lie
algebra $\mathfrak g\otimes \mathbb C$. The compact simply connected Lie
group integrating $\mathfrak h$ is denoted by $H$. Even if the Lie
algebra $\mathfrak h$ is real we need to use the complexified action
${\rho}^c$ since it is identified as subalgebra of $\mathfrak g\otimes
\mathbb C$ and has as action the induced action of ${\rho}^c$.

 By virtue of Proposition 2.1 and Proposition 2.2 in \cite{guilleminsternberg}  the analytic Lie algebra
 action induced on $\mathfrak h$ (which we denote by $\rho^c_{\mathfrak h}$ ) can be integrated to a local Lie group
 analytic
 action of the compact Lie group $H$, $\rho_H$. Consider the
averaged vector field with respect to a Haar measure on $H$: $$X_t^{c,H}=\int_G {{\rho^{c}_{H}}_{h}}_{*}(X_t^c) d\mu.$$

Observe that, by construction, the vector field $X_t^{c,H}$  is
invariant by the action of the Lie group $H$, $\rho^c_H$. Let $\phi_t$ be its flow ($ X_t^{c,H}(\phi_t^c(q))= {\partial \phi_t^c \over
\partial t} (q)$). It is an analytic diffeomorphism that preserves the
action of $\rho^c_H$. Thus it also preserves the action $\rho_{\mathfrak
h}^c$.

 Since $\phi_t^c$ is analytic and the lie algebras
 $\mathfrak h$ and $\mathfrak g$ have the same complexification, the family $\phi_t^c$ also
preserves the action of $\mathfrak g\otimes \mathbb C$,  $\rho^c$.

Let $X_t^H=Re{X_t^{c,H}}$ the real part of the vector field
$X_t^{c,H}$. Consider the real vector field $Y_t^H (x,y)=X_t^H(x,y,0,0)$ (we
use here the convention $(x,y)+i(y,t)$ is the complexification of
$(x,y) $ with $(x,y)=(x_1,\dots,x_n,y_1,\dots,y_n)$ and
$(z,t)=(z_1,\dots,z_n,t_1,\dots,t_n)$. By construction this vector
is invariant by the action of $\rho$. Now consider $\phi_t$  ($\phi_0=Id$) the family of analytic
diffeomorphisms satisfying the equation:\begin{equation}
\label{eqn:autonomous} Y_t^H(\phi_t(q))= {\partial \phi_t \over
\partial t} (q).\end{equation} The diffeomorphism   $\varphi_t:=\phi_t$ preserves the action $\rho$ and
thus preserves the linearity of the representation $\rho$.

In order to check that ${\varphi_t}^*{\omega_t}=\omega_0$, following Moser's trick, it is enough to check that the vector
field $Y_t^H$ satisfies the cohomological equation $i_{Y_t^H}\omega_t= -\beta$ where $d\beta=\omega_1-\omega_0$.

So as to do that, consider  the complexification of the
vector field $X_t$ that we denote by  $X_t^c$ which fulfills,\begin{equation}
\label{eqn:holmoser}
 i_{X_t^c} \omega_t^c=-\alpha^c
\end{equation}
\noindent where $X_t^c$, $\omega_t^c$ and $\alpha_t^c$ are the
complexification  of $X_t$, $\omega_t$ and $\alpha$
($d\alpha=\omega_1-\omega_0$) respectively.

Since $\omega_0$ and $\omega_1$ are invariant by the
linear  action $\rho$. The complexified forms $\omega_0^c$ and
$\omega_1^c$ are invariant under the action $\rho^c$ and  the path
$\omega_t^c$ is invariant by $\rho^c$. This path is also invariant
by the lifted Lie group action $\rho_H$.

Therefore from equation \ref{eqn:holmoser},  for each
$h\in\mathfrak h$ we obtain, $i_{{\rho_H{h}}_*{X_t^c}}\omega_t^c=-{{\rho_{H}}_{h}}_*(\alpha^c).$
Now averaging yields, $i_{X_t^{c,H}}\omega^c_t=-\int_H{{\rho_{H}}_{h}}_*(\alpha^c)$
here the new 1-form $\beta^c=-\int_H{{\rho_{H}}_{h}}_*(\alpha^c)$
still satisfies $d\beta^c=\omega_1^c-\omega_0^c$ due to
$\rho_H$-invariance of $\omega_0^c$ and $\omega_1^c$. For the vector field $Y_t^H$ defined by equation \ref{eqn:autonomous}, we obtain
$i_{Y_t^H}\omega_t= -\beta$ \noindent and by Moser's trick  the diffeomorphism $\phi_t$ defined in
\ref{eqn:autonomous} preserves the linearity of the representation
and takes $\omega_t$ to $\omega_0$. This ends the proof of the theorem.
\end{proof}

\section{The smooth case}

\subsection{The counterexample of Cairns and Ghys}

In this section we recall the results of Cairns and Ghys concerning
a $\mathcal{C}^{\infty}$-action of  $SL(2,\mathbb R)$ on $\mathbb
R^3$ which is not linearizable. All results mentioned in this
section are contained in section  $8$ of \cite{cairnsghys}.

Consider the basis $\{X,Y,Z\}$ of $\mathfrak{sl}(2,\mathbb R)$
satisfying the relations:
\[ [X,Y]=-Z,\quad  [Z,X]=Y,  \quad[Z, Y]=-X\]
Now consider the representation on $\mathbb R^3$ defined on this
basis as: \begin{equation}\label{eqn:linear}
\begin{array}{rlll}
\rho(X)&=& y\frac{\partial}{\partial z}+z\frac{\partial}{\partial y}  \\
\rho(Y)&=&x\frac{\partial}{\partial z}+z\frac{\partial}{\partial x} \\
\rho(Z)&=&x\frac{\partial}{\partial y}-y\frac{\partial}{\partial x}\\
\end{array}
\end{equation}

For this action the orbits are the level sets of the quadratic form
$x^2-y^2-z^2=r^2-z^2$ (where $r^2=x^2+y^2$).

We consider the deformation of this action as follows: Let
$R=x\frac{\partial}{\partial x}+y\frac{\partial}{\partial y}+
z\frac{\partial}{\partial z}$ stand for the radial vector field and take, \begin{equation}\label{eqn:perturbed}
\begin{array}{rlll}
\widetilde{X} &= &\rho(X)+fR,  \\
\widetilde{Y} &= &\rho(Y)+gR,\\
\widetilde{Z} & =&\rho(Z)\\
\end{array}
\end{equation}
\noindent with $f=xA(z,\sqrt{(x^2+y^2)})$ and $g=-yA(z,\sqrt{(x^2+y^2)})$ with
 $A(z,r)=\frac{a(r^2-z^2)}{r^2}$ where $a:\mathbb R\rightarrow
 \mathbb R$ is any $\mathcal C^{\infty}$ function which is zero on
 $\mathbb R^{-}$ and bounded.

 Following \cite{cairnsghys}, these new vector fields
 $\widetilde{X},
\widetilde{Y},  \widetilde{Z}$ generate a Lie algebra isomorphic to
$\mathfrak{sl}(2,\mathbb R)$ and define complete vector fields (since
the function $a$ is bounded). Therefore the defined action
$\hat\rho$ integrates to an action of a Lie group covering
$SL(2,\mathbb R)$. Since $Z$ is left the same, this action has to be
an action of $SL(2,\mathbb R)$. As it is  checked in
\cite{cairnsghys}, the orbits of this action are of dimension 3
(since the vector fields generating the action are independent)
whenever the function $a$ does not vanish\footnote{From the construction the set of points where $a$ does not vanish coincides with the hyperbolic elements of $SL(2,\mathbb R)$.} and coincides with the
initial (linear) action at $C=\{(x,y,z), x^2+y^2\leq z^2\}$. This justifies that the action
cannot be linearizable since the orbits of the linear action do no
have dimension three.
\begin{rem} The example of Guillemin and Sternberg \cite{guilleminsternberg} goes through the
following guidelines. It is quite similar to the counterexample of Grant and Cairns the difference is that vector field $Z$ is not preserved by the perturbation so  it cannot be guaranteed that it lifts to $SL(2,\mathbb R)$.
 If we perturb the initial action of $\mathfrak{sl}(2,\mathbb R)$
to the non-linear action: \begin{align*}
\hat\rho(X) &=\rho(X)+\frac{xz}{r^2}g(r^2-z^2)R,\\
\hat\rho(Y) &=\rho(Y)-\frac{yz}{r^2}g(r^2-z^2)R,\\
\hat\rho(Z) &=\rho(Z)+g(r^2-z^2)R,
\end{align*}
where $R=x\frac{\partial}{\partial x}+y\frac{\partial}{\partial y}+
z\frac{\partial}{\partial z}$ is the radial vector field, and $g\in
C^\infty(\mathbb R)$ is such that $g(x)>0$, if $x>0$, and $g(x)=0$,
if $x\le 0$. The orbits of $\rho$ coincide with the orbits of $\rho$ inside the
cone $r^2=z^2$. Outside this cone, the orbits of $\rho(Z)$ spiral
towards the origin while the orbits of $\rho(Z)$ are circles. Hence,
$\rho$ is not linearizable.
\end{rem}

\subsection{The Hamiltonian counterexample}

In this section we prove the following proposition which gives an
example of Hamiltonian action with linear part of semisimple type
and not $\mathcal{C}^{\infty}$-linearizable.

\begin{prop} Let $\alpha$ stand for the Lie group action of
$SL(2,\mathbb R)$ on $\mathbb R^3$ generated by the vector fields,
\begin{align*}
\overline{X} &=\rho(X)+fR,  \\
\overline{Y} &= \rho(Y)+gR,\\
\overline{Z} & = \rho(Z)
\end{align*}
 with $f=xA(z,\sqrt{(x^2+y^2)})$ and $g=-yA(z,\sqrt{(x^2+y^2)})$ with
 $A(z,r)=\frac{a(r^2-z^2))}{r^2}$ where $a:\mathbb R\rightarrow
 \mathbb R$ is any $\mathcal C^{\infty}$ function which is zero on
 $\mathbb R^{-}$ and bounded.

 Consider the lifted action $\hat{\alpha}$ of $SL(2,\mathbb R)$
to $T^*(\mathbb R^3)$ then this action is Hamiltonian and $\mathcal
C^{\infty}$ non-linearizable.

\end{prop}

\begin{proof}

 Given an action $\rho:G\times M\longrightarrow M$,
 the lifted action $\hat{\rho}$  to the  cotangent bundle
of $M$ is defined as  ,
$\hat{\rho_g}(\beta)=\rho_{g^-1}^*(\beta)$ with $\beta\in T^*(M)$. In coordinates $(q,p)$ of the cotangent bundle $T^*M$, we may define it as
\begin{equation}\label{eqn:cotangentlift}
\hat\rho_g(q,p)=(\rho_g(q), (d\rho_g)_q(p))
\end{equation}
 The cotangent bundle $T^*M$ is a symplectic manifold endowed
with the symplectic form $\omega=d\theta$ with $\theta$ the Liouville one-form. The Liouville one-form can be defined intrinsically as follows
$$\langle \theta_p, v\rangle:= \langle p, d\pi_p(v)\rangle$$ \noindent with  $v\in T(T^*M), p\in T^*M$.

It is well-known that the lifted action $\hat{\rho}$  is Hamiltonian.  Consider  $p\in T^*M, x\in \mathfrak g$ and let $\pi$  be
the standard projection from $T^*M$ to $M$ and  denote by  $\xi_{v}$
the fundamental vector field on $M$ generated by $v$. The  Hamiltonian
function $\mu:T^*M\longrightarrow \mathfrak{g}^*$ is:
\begin{equation}\label{eqn:lift}
\langle\mu(p),v \rangle= \langle \theta_p ,\xi_v \rangle =\langle p,\xi_{v}(\pi(p))\rangle
\end{equation}
where in the last equality we have used the definition of Liouville one-form.

From \ref{eqn:lift} we may define the components of the moment map $\mu_i$ associated to the lift of the vector $\xi_{v_i}$ as $\mu_i=\langle\theta,\xi_{v_i}\rangle$.

In our case, let us take $M=\mathbb R^3$, $G=SL(2,\mathbb R)$ and as
action $\alpha$. Then its lifted action $\hat{\alpha}$,
is a Hamiltonian action with linear part of type $\mathfrak{sl}(2,\mathbb R)$. We will prove that the lifted action is not $\mathcal C^{\infty}$-linearizable.

There are two important well-known facts concerning the lifted action to the cotangent bundle endowed with coordinates $(q,p)$:
\begin{itemize}
 \item By construction, the lifted action leaves the
zero-section of $T^*(\mathbb R^3)$ invariant and for $q\neq 0$, the projection $\pi$  sends  the orbit by the action $\hat{\alpha}$ through the point $(q,p)$, $\mathcal{O}^{\hat{\alpha}}_{(q,p)}$ onto the orbit  through the base point $q$ for the original action $\alpha$, $\mathcal{O^{\alpha}}_{q}$. In particular,  $\dim \mathcal{O^{\hat{\alpha}}}_{(q,p)}\geq \dim \mathcal{O^{\alpha}}_{q}$.

\item From the equation \ref{eqn:cotangentlift} the action is linear on the fiber over a fixed point of $\alpha_g$. So for $q=0$ the action on the fibers is the dual to the linear action ${\alpha}^{(1)}$. In particular the action restricted to the fiber over $q=0$ has as orbits the origin ($p=0$) which is a $0$-dimensional orbit  and  $2$-dimensional orbits (for $p\neq 0$).

\end{itemize}

Applying the first property to the example $\hat{\alpha}$ we can conclude that for points that project to hyperbolic orbits on the base,  $\dim \mathcal{O^{\hat{\alpha}}}_{(q,p)}\geq \dim \mathcal{O^{\alpha}}_{q}=3$ (since from the example of  Cairns and Ghys \cite{cairnsghys} the orbits with  $a>0$ are of dimension $3$).

We denote by $\hat{\alpha}^{(1)}$ the linearization of the lifted action. In view of the remarks above, this action coincides with the linear action $\hat{\alpha}^{(1)}$ when restricted to the zero-section $p=0$ and to its dual $\hat{\alpha}^{(1)*}$ when restricted to $q=0$. In particular, the linear action restricted to the set $\Omega=\{p=0\}\cup \{q=0\}$  is formed by $2$ and $0$-dimensional orbits.

Let us denote by $S$ the set of  all $0$ and $2$-dimensional orbits of the lifted linear action. We can compute explicitly this set, by identifying the lifted vector fields of the action with the Hamiltonian vector fields with moment map components $\theta(X_i)$ with $X_i$ the generators of the action on the base. Denoting as Liouville one-form $\theta=adx+bdy+cdz$, the moment map  of the lifted action given by equations \ref{eqn:linear} is $\mu=(zb+cy,az+xc,-ay+bx)$. The dimension of the orbits can be computed using the rank of $d\mu$ which is generically $3$ and is $2$ in a union of manifolds described by the vanishing of a set  of minors (given by linear and quadratic homogeneous polynomials)\footnote{
 Identifying this action with the lifted coadjoint action of $SL(2,\mathbb R)$, we can indeed describe this set of $2$-dimensional orbits  in terms of the isotropy groups of the components of  $(q,p)$. We thank Marco Castrillón for enlightening our computations with this beautiful idea.}.

 Let us now come back to the lifted deformed action $\alpha$ defined by equations \ref{eqn:perturbed}.
The solid cone $C=\{(x,y,z), x^2+y^2\leq z^2\}$ on the zero section is saturated by parabolic and elliptic orbits of the linear action $\alpha^{(1)}$ (as defined in equation \ref{eqn:linear}) of $SL(2,\mathbb R)$. Its \lq\lq lifted\rq\rq cone, $\hat C$, is saturated by the orbits of the lifted linear action. Observe that, by construction, both actions $\alpha$ and $\alpha^{(1)}$ coincide on the lifted solid cone.

Now assume that there existed a diffeomorphism $\phi$ conjugating the actions $\hat{\alpha}^{(1)}$ and $\hat{\alpha}$, $\phi\circ\hat{\alpha}^{(1)}\circ \phi^{-1} = \hat{\alpha}$. The set $\hat C$ is formed by 0, 2 and 3-dimensional orbits of the lifted linear action. Consider, $S$ the set of $0$ and $2$-dimensional orbits by the lifted linear action.

 In case there existed a diffeomorphism $\phi$ conjugating the actions $\hat{\alpha}^{(1)}$ and $\hat{\alpha}$ then $\phi(S)=S\cap \hat{C}$ (since all the orbits by the action of $\hat{\alpha}$ outside $\hat{C}$ are 3-dimensional).
    Observe that the origin $O$--zero-dimensional orbit-- goes to the origin--zero-dimensional orbit-- but the neighbourhoods of $O$ in   and $S$ and in $\phi(S)$ cannot be diffeomorphic. Thus $\hat{\alpha}$
cannot be equivalent to the linear action $\hat{\alpha}^{(1)}$.
\end{proof}

\subsection{The case of semisimple Lie algebras of compact type}

When the Lie algebra action is of compact type, it can be integrated to an
action of a compact Lie group $G$ (see \cite{ruiphilippe} for a proof
of that fact is done using algebroids).

Given a fixed point for the action $p$,  we can associate a  linear
action of the group in a neighbhourhood of $p$ and the action of the group both preserving the
symplectic structure (which we can assume to be in Darboux coordinates). Apply the equivariant Darboux theorem
 \cite{chaperon} to obtain a diffeomorphims $\phi$  that linearizes the group action $G$ in Darboux
 coordinates. By differentiation we get linearization of the Lie algebra action $\rho$.

\section{ Linearization for semisimple actions on Poisson manifolds}

In this section we consider representations of Lie algebras by vector fields preserving a  Poisson structure. We present an equivariant Weinstein theorem for analytic actions of semisimple algebras on analytic Poisson manifolds.

A Poisson structure on a manifold $M$ is a bivector field $\Pi\in \Gamma(\Lambda^2 TM)$ satisfying the  equation $[\Pi, \Pi]=0$ where the bracket $[\cdot, \cdot]$, called Schouten bracket, is an extension of the Lie bracket for vector fields to bivectors. The first observation is that there are no dimensional (dimension may be odd) or topological constraints on a smooth/analytic manifold to admit a Poisson structure.

\subsection{The case of $b$-Poisson manifolds}

We start dealing with a special class of even dimensional Poisson manifolds which have been recently studied in \cite{guimipi12} called $b$-Poisson (or $b$-symplectic) manifolds. The study of this manifold was originally motivated for the study of calculus on manifolds with boundary \cite{melrose} and deformation quantization on symplectic manifolds with boundary \cite{nestandtsygan}.

\begin{defn}\label{definition:firstb}
Let $(M^{2n},\Pi)$ be a Poisson manifold such that the map
$$p\in M\mapsto(\Pi(p))^n\in\Lambda^{2n}(TM)$$
is transverse to the zero section, then $Z=\{p\in M|(\Pi(p))^n=0\}$ is a hypersurface and we say that $\Pi$ is a \textbf{$b$-Poisson structure} on $(M,Z)$ and $(M,Z)$ is a \textbf{$b$-Poisson manifold}.
\end{defn}

As it is seen in \cite{guimipi12}, the class of $b$-Poisson manifolds shares many properties with the class of symplectic manifolds. The reason for this is that a $b$-Poisson structure can be studied using the language of forms and the path method works also well in this category. More concretely, using the language of $b$-forms instead of the language of bivectors native to Poisson geometry.

 A $b$-form of degree $k$ is a section of the bundle $\Lambda^{k}(^b T^*M)$ where the bundle $^b T^*M$ is defined by duality $^b T^*M= (^b TM)^*$. The bundle $^b TM$ is defined à la Serre-Swan as the bundle whose sections are vector fields on $M$ which are tangent to the critical hypersurface $Z$ (for more details refer to \cite{guimipi12}). A $b$-symplectic structure is  a closed  $b$-form  of degree $2$ that is non-degenerate (i.e of maximal rank as an element of $\Lambda^2(\,^b T_p^* M)$ for all $p\in M$). One may assign a $b$-symplectic structure to a $b$-Poisson structure and viceversa \cite{guimipi12}.

As a first application of the path method for $b$-symplectic forms we obtain a $b$-Darboux theorem (for a proof see \cite{guimipi12}),

\begin{prop}[\textbf{b-Darboux Theorem}]\label{prop:localb}
Let $(M,Z)$ be a $b$-Poisson manifold, with Poisson bivector field $\Pi$ and dual two-form $\omega_\Pi$. Then, on a neighborhood of a point $p\in Z$, there exist coordinates $(x_1,y_1,\dots x_{n-1},y_{n-1}, z, t)$  centered at $p$ such that
$$\omega_{\Pi}=\sum_{i=1}^{n-1} dx_i\wedge dy_i+\frac{1}{z}\,dz\wedge dt.$$
\end{prop}

We recall from \cite{btoric} the following $b$-symplectic linearization theorem  for actions of a compact Lie group $G$ which is proved using the path method for $b$-forms.

\begin{thm}[\textbf{equivariant b-Darboux}, Guillemin-Miranda-Pires-Scott]\label{thm:equivDTgeneralG}
Let $\rho$ be a $b$-symplectic action of a compact Lie group $G$  on a $b$-symplectic manifold $(M,Z,\omega)$, and let $p \in Z$ be a fixed point of the action.  Then there exist local  analytic coordinates $(x_1,y_1,\dots, x_{n-1},y_{n-1}, z,t)$ centered at $p$ such that the action is linear in these coordinates and,
\[
\omega =\sum_{i=1}^{n-1} dx_i\wedge dy_i+\frac{1}{z}\,dz\wedge dt.
\]
\end{thm}

This setting was proved for smooth forms but works in the analytic case too.
We can now apply the same trick as in the symplectic case which we saw in detail in section 2 and complexify the action to take the compact real part and apply theorem \ref{thm:equivDTgeneralG}. This yields,

\begin{thm} Let $\mathfrak g$ be a semisimple Lie algebra and let $(M,\omega)$ be a (real or complex) analytic $b$-symplectic
manifold $(M,Z,\omega)$. Let   $\rho: \mathfrak g\longrightarrow L_{analytic}$  be a
representation by  analytic  vector fields preserving  the $b$-symplectic structure.   Then there exist local coordinates $(x_1,y_1,\dots, x_{n-1},y_{n-1}, z,t)$ centered at a fixed point  $p \in Z$ for $\rho$ such that the action is linear in these coordinates and
\[
\omega =\sum_{i=1}^{n-1} dx_i\wedge dy_i+\frac{1}{z}\,dz\wedge dt.
\]
\end{thm}

\subsection{ The general Poisson case}

We start by recalling the equivariant Weinstein theorem for smooth Poisson structures, proved in \cite{mirandazung2006} for tame Poisson structures in \cite{mirandamonnierzung} for Hamiltonian actions and recently by Frejlich and Marcut \cite{transversals} in full generality,

\begin{thm}[\textbf{equivariant splitting theorem}, Miranda-Zung, Frejlich-Marcut]
Let $(P^n,\Pi)$ be a smooth Poisson manifold, $p$ a point of $P$,
$2k = \rm rank \ \Pi (p)$, and $G$ a semisimple compact Lie group
which acts on $P$ preserving $\Pi$ and fixing the point $p$.  Then
there is a smooth canonical local coordinate system
$(x_1,y_1,\dots,x_{k},y_{k},$ $ z_1,\dots, z_{n-2k})$ near $p$, in
which the Poisson structure $\Pi$ can be written as
\begin{equation}
\Pi = \sum_{i=1}^k \frac{\partial}{\partial x_i}\wedge
\frac{\partial}{\partial y_i} + \sum_{ij}
f_{ij}(z)\frac{\partial}{\partial z_i}\wedge
\frac{\partial}{\partial z_j},
\end{equation}
with $f_{ij}(0)=0$, and in which the action of $G$ is linear and preserves the subspaces
$\{x_1 = y_1 = \hdots x_k = y_k = 0\}$ and $\{z_1 = \hdots =
z_{n-2k} = 0\}$.
\end{thm}

The difference between the approaches to the proof in \cite{mirandazung2006}
and  \cite{transversals} is that in the latter the authors introduce a technical important tool in Poisson geometry called {\bf{Poisson transversals}}, whereas in \cite{mirandazung2006} the approach was done using Vorobjev data \cite{vorobjev}.

In both technical approaches to the proof, there are horizontal data corresponding to the symplectic form on a leaf and transversal data. In \cite{vorobjev} the transversal data is the Poisson structure induced on a transverse vector field. In \cite{transversals} the choice of horizontal and vertical spaces is done carefully using sprays associated to adapted transversals. One of the important achievements of this last method is that we can apply the path method, changing the vertical part with a full control of the smooth/analytic properties.
 Following the same proof in \cite{transversals} replacing smooth sprays by analytic sprays, we can apply the analytic symplectic linearization theorem \ref{thm:main} in the horizontal direction  and theorem \ref{thm:gs} in the transversal direction to obtain,

\begin{thm} Let $\mathfrak g$ be a semisimple Lie algebra and let $(M,\omega)$ be a (real or complex) analytic $(P^n,\Pi)$ Poisson manifold and  $p$ a point of $P$ with $2k = \rm rank \ \Pi (p)$. Consider $\rho: \mathfrak g\longrightarrow L_{analytic}$  a
representation by  analytic  vector fields preserving $\Pi$ and fixing $p$.
   Then
there exists an analytic canonical local coordinate system
$(x_1,y_1,\dots,x_{k},y_{k},$ $ z_1,\dots, z_{n-2k})$ near $p$, in
which the Poisson structure $\Pi$ can be written as
\begin{equation}
\Pi = \sum_{i=1}^k \frac{\partial}{\partial x_i}\wedge
\frac{\partial}{\partial y_i} + \sum_{ij}
f_{ij}(z)\frac{\partial}{\partial z_i}\wedge
\frac{\partial}{\partial z_j},
\end{equation}
with $f_{ij}(0)=0$, and in which the representation $\rho$ is linear and preserves the subspaces
$\{x_1 = y_1 = \hdots x_k = y_k = 0\}$ and $\{z_1 = \hdots =
z_{n-2k} = 0\}$.
\end{thm}

\end{document}